\documentclass{article}
\usepackage{amsthm}
\begin{document}
\newtheorem{defn}{Definition}
\newtheorem{theorem}{Theorem}
\newtheorem{lemma}{Lemma}
\newtheorem{cor}{Corollary}
\newtheorem{prop}{Proposition}
\newtheorem{conj}{Conjecture}
\title{Bergman-Szeg\H{o} asymptotic formulas and the strip problem}
\author{Mark G. Lawrence}
\maketitle
\section{Introduction}
In an important paper, \cite{Tu},  Tumanov showed that if a continuous function $f$ on the closed strip  strip $ S =\{z=x+iy: |y|\le 1\}$ has the property that if the restriction of $f$ to the circle $|z-t|=1$ extends holomorphically to $|z-t|<1$ for all $t\in \bf R$, then 
$f$ itself must be holomorphic.
The proof involves complexification. One wants to study the collection of all holomorphic extensions as one object; by lifting to $\bf C^2$ in a canonical way, one obtains a CR function on a Levi-flat CR manifold. The details are discussed below in the proofs in this paper. The method of complexification for these problems was explored earlier in papers of Agranovsky and Globevnik \cite{AgGl}

The general type of problem is called the study of the 1-dimensional extension property. In the strip problem, you want to know if a function is holomorphic. One can also ask if a function defined on a boundary in $\bf C^n$ has a holomorphic extension based on extensions on a thin set of slices. A very general paper which covers questions of both types, for real analytic functions, is due to Agranovsky \cite{Ag}.

Our interest in the strip problem is for two reasons. First, this is a simplified relative of the CR Hartogs problem, \cite{La4}, which is to tell when a function defined on a convex boundary in $\bf C^n$ has a holomorphic extension, based on holomorphic extension on slices by affine lines parallel to the coordinate axes. 
The second reason was our hope that by proving a theorem for $L^p$ functions, one could discover new function theory.
In this paper, we show how to realize the second goal in part. 
Our main result is to show that a weighted Bergman projection on the strip has an asymptotic expression involving Szeg\H{o} projections on sliding curves. This theorem is valid for rather general perturbations of ellipses (the same as in \cite{La4}). 
We also fill in a gap in the literature by proving a weighted $L^p$ strip theorem for circles. Finally, we suggest further directions for study as well as some open questions. 

The strip problem is quite specialized, so it's fair to ask why attention is focused on it. We remark that the theorems here easily translate to conformal annuli, where the curve family in the annulus is the image of a family on the strip by the universal covering map.  For the author, the strip with translated curves is a model case where computation is relatively straightforward. A general theorem for planar domains would involve a nice family of simple closed curves, where the family is 1-dimensional, but every point in the domain is contained in two different curves. In this case we expect that an analogue of the strip theorem is true. In order to prove our theorems about certain weighted Bergman projections on a  domain, it turns out that we need the analog of the strip theorem---i.e. that if a function extends holomorphically to the interior of each simple closed curve in the given family. There is an unknown issue about smoothness. Everything in the literature that the author is aware of depends on real analytic families of curves, and often algebraic. N.B. for the strip problem one could state a theorem for real analytic perturbations of an ellipse without much difficulty.
Is the strip theorem true for a sliding $C^2$ smooth strictly convex curve? One would guess yes, possibly requiring $C^k$ smoothness for some $k>2$. Methods would have to be completely different. 

In order to get the new results, certain technical improvements to \cite{La3} are required. Once the improved strip theorem is known, the main theorems follow rather easily. In particular the  Bergman-Szeg\H{o} formula is a simplified version of Corollary 1 in \cite{La4}.

\begin{defn} Let $L^p(S,\alpha)=\{f(z), z\in S: \int_S|f(z)|^p\frac{dA(z)}{(1-y^2)^\alpha}<\infty\}$ , $p>1$. 
\end{defn}
Then $f\in L^p(S,p,\alpha)$ iff $\int_{-\infty}^{\infty}(\int_{C_t}|f(z)|^p\frac{ds}{(1-y^2)^{\alpha-\frac{1}{2}}})dt<\infty$. 
One can also define $L^p_{loc}(S,\alpha)$, where "loc" means in $\bf C$---i.e. $\int_{-M}^M\dots <\infty$ for all $M$. 
To prove a strip theorem, it suffices to work in $L^p_{loc}(S,\alpha)$. For the applications to function theory one needs $L^p(S,\alpha)$. 

For $p>2$, the strip problem is known to have a positive solution for functions in 
Here is our improvement, which includes the original theorem. The curve $C$ here is either a curve for which \cite{La3} applies, or the circle. $C_t$ is the translate by $t$ real units.

\begin{theorem}
Let $f$ be in $L^p(S,\alpha)$, $p>2-\alpha$, $\frac{1}{2}\le \alpha <1$. 
If for almost every $t\in \bf R$, 
$f|_{C_t}$ extends holomorphically to $|z-t|<1$, then $f$ is holomorphic.
\end{theorem}

\section{New proofs of the strip problem and a proof for the sliding circle}
The geometry of analytic continuation for the case of the circle is different than for the ellipse---easier, in fact, but different. In this section we present the modification of the strip problem for different weights and show how to deal with the case of the circle as well. This allows for a unified treatment of Bergman-Szeg\H{o} formulas in the later sections.

Now if $f\in L^p_{loc}(S,\alpha)$  and $f|_{C_t}$ has a holomorphic extension to $D_t$ for a.e. $t$, then one can adapt standard growth estimates for Hardy space functions. Set $\beta=\alpha-\frac{1}{2}$. Set $L_p(C_t,\beta)$ be the weighted $L^p$ space with norm $||f||_{p,\beta}=(\int_{C_t}|f(z)|^p\frac{ds}{(1-y^2)^{\beta}})^{1/p}$.

Here is the proof of Theorem 1. 

The proof follows the outline of \cite{La3} with the different  function spaces being used to improve Lemma 5. 
For perturbations of an ellipse, this is the only change. For the circle, analytic continuation in $\bf C^2$ which already appeared in \cite{Tu} needs to be justified for weighted integrable functions. 

For any  choice of $C$ as in the theorem, the following construction holds. Let $D_t$ denote the domain bounded by $C_t$. 
Let $C$ be defined by $P(z,\overline z)=0$.

We construct a singular CR manifold $M$ as follows. First set 
$$K=\cup_t\{(z,w): z\in D_t, P(z-t,w-t)=0\}.$$ For $z\in S$, let $K_z=\{w: (z,w)\in K\}$.  For each $y$, $0<y<1$, there is a $t_y$ such that for $z=x+iy$
$K_z$ is the union of solutions of $P(z-t,w)$ for $x-t_y\le t\le x+t_y$. 
in Theorem 2 of \cite{La3} it holds that when $z=x+iy$, $y\ne 0$.  $K_z$ is the union of continua $\gamma_j(z)$, $1\le j\le n$, where $n$ is independent of $z$; $\gamma_1$ is a simple closed curve with  $\overline w\in \gamma_1(t)$; and no $\gamma_j$, $j\ne 1$ crosses $\gamma_1$ for any value of $z$. Define $M$ by its cross-sections: $M_z=\gamma_1(z)$.

Finally, a measurable $CR$ function $F$ is defined on $M$ by setting $F(z,w)=f_t(z)$, where  $t$ is such that $P(z-t,w)=0$. By construction, from the assumptions of the theorem, this $t$ is unique, so $F$ is well-defined $a.e.$ on  $M$.

The proof of the theorem depends on two steps. First, a Lewy-type extension theorem is proved, showing that $F$ extends holomorphically to the domains $\Omega_{\pm}$, where $\Omega_+$ is the domain obtained by filling in the $\gamma_1$'s for each $z$ in the upper half strip, with $\Omega_-$ the same, for the lower half-strip. 
The second step is to use some analytic continuation methods which result in proving that $F$ does not depend on $w$, which means the original $f$ was holomorphic. 

For the first step, the only difference will be in Lemma 5, which we repeat now. 

\begin{lemma} Let $\gamma$ be smooth, strictly convex curve bounding the region $G$. Suppose $t=\min_{\gamma} Im(z)$ is obtained at $z=0$. Then $$\int_{G\cap (Im (z)=t)}\frac {1}{(dist(z, \partial G))^{1/p}}dx\rightarrow 0$$ 
as $t\rightarrow 0^+$ if and only if $p>2$. 
\end{lemma}

The lemma was applied in the following way. We show how to get the needed estimates at the top of the strip. The same method applies at the bottom.
Let $L_t=||f_{C_t}||_p$, then on $D_t$, 
$$|f(z)|\le \frac{L_t}{(dist(z,\partial D_t))^{1/p}}.$$

For $f\in  L^p_{\alpha, Loc}$, 
$$\int_{C_t}\frac{|f(z)|^p}{|i-(z-t)|^{\alpha -1/2}}|dz|\le \int_{C_t}\frac{|f(z)|^p}{|1-y|^{\alpha -1/2}}|dz|.$$

Thus, for almost all $t$, 
$\frac{f_{C_t}(z)}{(i-(z-t))^{\frac{2\alpha-1}{2p}}}  \in L^p(C_t)$ 
with $L_p$ norm bounded by $||f||_{p,\alpha,t}$. 

Using standard estimates for growth of $H^p$ functions, we obtain that on $D_t$,
$$|f(z)|\le \frac{L_{\beta, t}(f)|i-(z-t)|^{\beta/p}}{(dist(z,\partial D)^{1/p}},$$
where we have set $\beta=\alpha-1/2$.
Lemma 5 is applied by flipping  $C_t$ over, sending the top to 0.
Then, following the proof in \cite{La} we need to evaluate $\int_0^{\sqrt t}\frac{(x^2+t^2)^{\beta/2p}}{\sqrt t (\sqrt t-x)^{1/p}}dx$. 
The numerator can is bounded by $Kt^{\frac{\beta}{2p}}$, which comes out of the integral. Thus, end up with the integral being bounded by $Kt^{\frac{1}{2}+\frac{\beta}{2p}-\frac{1}{p}}$, which tends to 0 as $t\rightarrow 0$ when $p>2-\beta$. 
This estimate allows the proof of analytic continuation to $\Omega_{\pm}$ to go through for the new case. 

For $C$ which is the perturbed ellipse, the rest of the proof from \cite{La3} goes through essentially unchanged---the functions have better estimates, so a fortiori the original proofs work. We remark that this part of the proof uses $L^1$ independent of $p$. The same is true for the circle as outlined below.
Now for the case of the circle:

The original strip theorem was proved for $S=\{z:|y|<1\}$, $C_t=\{z:|z-t|=1\}$, for functions continuous on the closed strip.
The method of complexification used there is what the author modified for his own work. What is written here  about  is  equivalent to Tumanov's proof, although the presentation is different. 
Let $V_t=\{(z,w):0<|z-t|<1, w=\frac{1}{z-t}+t\}$ and let
$M=\cup_{t\in {\bf R}}V_t$. The fiber $M_z$ over $z\in S$ is a simple closed curve except for $z\in {\bf R}$ when it is the real axis. 
Let $\Omega_z$ be the domain obtained by filling in $M_z$ for $z\notin \bf R $. As in the case of the ellipse, the proof has two parts.  First, one shows that the lifted CR function $F$ extends holomormophically to the $\Omega=\cup \Omega_z$. This is done by following a Hans Lewy argument. This step goes the same for the circle as for the ellipse, including in the new, weighted case. 

The second step is to apply an analytic continuation argument to obtain a  holomorphic functions defined in domains above a.e. $x\in \bf R$. 
 As $z$ tends to the real axis from above and below, one gets a limiting domain, which when filled in across the real axis in the vertical fiber, is actually the whole plane. Analytic continuation across the real axis is true because Cauchy's theorem holds. For a continuous function, one then uses Liouville's theorem to show the function is constant on almost every fiber over $\bf R$, then a boundary value plus uniqueness arguments shows that the  holomorphic function on $\Omega$ is independent of $w$. This finishes a proof of the strip theorem. Instead of Liouville's theorem, we show functions are in certain $H^1$ spaces on the upper and lower half planes, and from this deduce the functions are constant. 
The second step in the proof of the strip theorem does not use the stronger weight on the boundary of the strip. 

Here are the details. For each $z=x+iy\in S$ which is not real, 
$M_z$ is parametrized by $w=\frac{1}{z-t}+t, -\sqrt{1-y^2}+x<t<x+\sqrt{1-y^2}$. This is a simple closed curve. In the case of ellipses and suitable perturbations, branching is crucial for understanding what happens. Here the situation is simpler.
Let $\Omega$ denote the open set in $Im(z)\ne 0$ which is obtained by filling in the $M_z$'s. 
Following \cite{La3} we get analytic continuation to $\Omega$ for $f\in L^2_\alpha(S)$.
The key estimates are for $$\int_{x=a}^{x=b}\int_{x-\sqrt{1-y^2}}^{x+\sqrt{1-y^2}}|F(x+iy),t|dtdx.$$
Exactly the same estimates hold for the case of the circle. The difference is relating these estimates to integrals over $M_z$'s. 
For the ellipse case of the strip problem, $|dw|$ on $M_z$ and $dt$  are comparable. For the case of the circle $\frac{|dw|}{1+v^2}$ and $|dt|$ are comparable. As $y\rightarrow 0$ from above or below, the fibers open up into the real axis, with the inside going to the lower and upper half places for $y>0$ and $y<0$ respectively. 
For the ellipse moments with $w^n$ were evaluated. For the circle we use moments adapted to the half plane. For the upper half plane, use $(\frac{w-i}{w+i})^n$, and the reciprocal for the lower half plane.
As in \cite{La3}, dominated convergence then shows that 
for almost every $x$, $\int_{-\infty}^{\infty}|F(x,t)|\frac{dt}{1+t^2}<\infty$.
By using conformal maps to the interiors of $\gamma(t)$, each sending 0 to $i$, we can show that moment condition
$$\int_{-\infty}^{\infty}F(x,t)(\frac{t-i}{t+i})^n \frac{dt}{1+t^2}=0$$
holds for almost every $x$. 

These moment conditions guarantee that $F(x,t)\in H^1$, where $H^1$ is with respect to the measure $\frac{dt}{1+t^2}$. 
The functions on the upper and lower half planes glue together to become an entire function.
Finally, there is no entire non-constant entire function which is in this $H^1$ on both the lower and upper half planes.

\section{Bergman projection adapted to the strip problem}
For this section, $C$ denotes a smooth strictly  convex curve in $|Im(y)|\le 1$
whose top and bottom points touch the top and bottom of the closed strip.
$C$ should  be symmetric with respect to the $x$-axis, and have a symmetry about a vertical axis as well. 
Let $C_t$, $C_t^+$, $C_t^-$ denote, respectively, the horizontal translate of $C$ by $t$ units,  and the right and left  halves of $C_t$.
Then we can relate integration over the curves and the strip as previously.
\begin{lemma} There is a smooth positive function $\lambda(y)$ on $[-1,1]$ such that\
$$\int_S f(z)\frac{\lambda dx dy}{\sqrt{1-y^2}}=\int_{-\infty}^{\infty}\left(\int_{C_t^+} f(z)ds\right)dt=\int_{-\infty}^{\infty}\left(\int_{C_t^-}f(z)ds\right)dt.$$
Here $ds$ is arc length. 
\end{lemma}
 
 Next we define Szeg\"{o} and partial Szeg\"{o} operators. The latter is a new concept which is adapted to the strip problem. 
 
 For each $t$, $T_t$ is the Szeg\"{o} projection in $L^2(C_t, \frac{ds}{(1-y^2)^\beta})$. 
 $T_t^{\pm}f$ denotes the restriction of $T_tf $ to $C_t^{\pm}$.
 At a point $z$, there are  two values of $t$, $t_1$ and $t_2$ corresponding to translates of $C$ containing $z$.
 Set $\pi f(z)=\frac{1}{2}(T_{t_1}f(z)+T_{t_2}f(z))$.
 
 \begin{lemma} 
 $\pi$ is a self adjoint operator  on $L^2(S, \frac{\lambda(y) dx dy}{(1-y^2)^{\alpha}})$. 
\end{lemma}

 \begin{proof}
 Let $f,g\in L^2(S, \frac{\lambda(y) dx dy}{(1-y^2)^{\alpha}})$
 
 Then $$\int_S (\pi f)g \frac{\lambda(y)dx dy}{(1-y^2)^{\alpha}}=$$
 $$\int_{-\infty}^{\infty}\left(\int_{C_t} (T_tf(z))\overline{g(z})\frac{ds}{(1-y^2)^{\beta}}\right)dt=$$

 $$\int_{-\infty}^{\infty}\left(\int_{C_t} f(z)\overline{ T_t(g(z)}
\frac{ds}{(1-y^2)^{\beta}}\right)dt=$$
 
 $$\int_S f(z)\overline{\pi g(z)}\frac{\lambda(y) dx dy}{(1-y^2)^{\alpha }}.$$
 \end{proof}
 
 Obviously $||\pi||=1$. The norm is attained for analytic functions. 
 If any function is not analytic, then the norm of the Szeg\"{o} projection on some $C_t$'s is diminished, which decreases the norm of $\pi f$.
 We are using the truth of the strip theorem in this step.
 By a simple application of the spectral theorem, we have the following expression for the Bergman projection. Define $B_{\alpha}$ to be the Bergman projection operator for  $L^2(S, \frac{\lambda(y) dx dy}{(1-y^2)^{\alpha}})$. 
 \begin{theorem}
 Let $f\in L^2(S, \frac{\lambda(y)dx dy}{(1-y^2)^{\alpha}})$. Then $$\lim_{n\rightarrow \infty} \pi^nf\rightarrow B_{\alpha},f$$ with convergence in the strong operator topology. 
 \end{theorem}
 
 The usual expressions for weighted Bergman projections are area integrals, of course. Here the expression for $\pi$ at a point involves only the Szeg\H{o} projection on two curves. As the power of $\pi$ increases, there is an averaging over curves which cover a larger and larger part of the strip. It's possible there is a singular integral expression for the operator $T^n$. If so, one might be able to prove convergence in $L^p, 1<p<\infty$. Our current methods do not reach this goal.

\section{Higher dimensional analogues of the strip theorem and Bergman-Szeg\H{o} formulas}

In $\bf C^{n+1}$ , $n>0$ with coordinates $(z_1, z'), z_1\in \bf C\}$,  consider a sliding family of smooth strictly convex domains, 
$D_t=\{z|z-(t,0)\in D_0\}$. Let $\Omega=\cup_t D_t$. The question is, given a function $f$ on $U$ such that $f|_{\partial D_t}$ extends holomorphically to $D_t$ for all  $t$ (or almost all if $f$ is integrable in some class), does $f$ have to be holomorphic. The first observation is that if $f\in C^1(U)$, the conclusion follows easily.
At a generic point of $p \in U$, the two $\partial D_t$'s which contain $p$ are transverse. Since $n>0$, we have Cauchy-Riemann equations holding on tangential directions. Since those directions span, we get the Cauchy-Riemann equations holding in all directions, so $f$ is holomorphic. 
This is in marked contrast to the strip theorem for $\bf C$, where even for real analytic functions the theorem is non-trivial. 
There's not an obvious way to go from the $C^1$ case to the weighted $L^p$ space directly, but by slicing we can apply the 1-dimensional theorem, for suitable $D$, after which a Bergman-Szeg\H{o} formula will hold automatically. 

Here is a theorem with spheres.

\begin{theorem} Let $z=(z_1, z'), z_1\in {\bf C}, z'\in {\bf C^{n-1}}$.  Let $B_t=\{|z-(t,0)|^2<1\}, S_t=\partial B_t$ and define 
$\Omega=\cup_t B_t$. With $dm$ the usual Lebesgue measure on $\Omega$,  set $d\omega_{\alpha}=\frac{dm}{(1-(y_1)^2 +|z'|^2))^{\frac{1}{2}+\alpha}}, \frac{1}{2}<\alpha<1$.
Suppose $f\in L^p(\Omega, d\omega_{\alpha}), p>2-\alpha$ extends holomorphically from $S_t$ to $B_t$ for a.e. $t$. 
Then $f$ is holomorphic on $\Omega$. 
\end{theorem}

Also, for $p=2$, you get a Bergman-Szeg\H{o} formula. 

\begin{proof}
 On slices of the form $z'=c$, the strip theorem for $\bf C$ applies. 
On each $B_t$, there is a holomorphic extension from $S_t$. Because the strip theorem applies on the slices, this holomorphic function is actually equal to $f$ on each $B_t$, so is holomorphic on $\Omega$.
\end{proof}

The next example is in the plane, but is derived from the author's CR Hartogs theorem. 
Set $\rho(z,w)=|z|^2+|w|^2-1+\epsilon(zw+\overline{zw})$ for some small $\epsilon>0$ and define $D=(\rho<0)$. 
Here is the theorem we need; it is covered by the main theorem of \cite{La4}. 
\begin{theorem} Let $f\in L^2(\partial D)$ and suppose that for almost every slice $(z=c)\cap\partial D$ and almost every slice $(w=k)\cap \partial D$, $f$ extends holomorphically to the corresponding slice of $D$. Then $f$ is CR---it extends holomorphically to $D$. 
\end{theorem}

In the CR Hartogs paper there is a corollary  about the Szeg\H{o} projection analogous to the Bergman-Szeg\H{o} construction, but we don't use it.

For a volume form on $\partial D$, we use 
$dV=\frac{dz}{z+\epsilon\overline w}\wedge dw \wedge d\overline{w}=\frac{dw}{w+\epsilon\overline z}\wedge dz\wedge d\overline{z}$. 
 
Let $\pi:L^2(\partial D)\rightarrow L^2(\partial D)$ be the orthogonal projection onto the space of functions which have holomorphic extensions on almost all slices $w=c$. Let $T$ be the operator on $L^2(\partial D)$ coming from $f(z,w)\rightarrow f(w,z)$. 

We need the following facts.
\begin{enumerate}
\item $\pi(f)$ is computed by taking the Szeg\H{o} projection on slices. This is because of the product form of the measure.
\item $T$ is self-adjoint and $T\circ \pi=\pi\circ T$. The first is obvious and the second follows from the symmetry of the domain. 
\item If $f(z,w)=f(e^{i\theta}z, e^{-i\theta}w)$ then $\pi(f)$ has the same symmetry. 
\end{enumerate}

For $\phi(zw)=zw$, let $G=\phi(D)$. 
We will demonstrate a Bergman-Szeg\H{o} formula for $G$, but it's easier to carry out the computations on $D$ directly.
\begin{prop} Let $S=\pi\circ(\frac{1}{2}(T+I))$.
Then for any $f\in L^2{D}$ satisfying $f(z,w)=f(e^{i\theta}z, e^{-i\theta}w)$, 
$$\lim_{n\rightarrow \infty}S^n(f)=g,$$
where $g$ is the orthogonal projection of $f$ onto the CR functions which are invariant under $(z,w)\rightarrow(e^{i\theta}z, e^{-i\theta} w)$. 
\end{prop}

The proof is essentially the same as for the Bergman-Szeg\H{o} formula. Now we interpret this for functions on $G$.
We start with a function $f$ on $G$. After lifting it to $\partial D$ and applying $S$, we get a function of $zw$ again which we can think of as a function on $G$.
Some elementary calculations show the following.
Let $\cal C$ be the collection of curves in $G$ which are obtained by projecting the slices $(w=c)\cap \partial D$.  Then 
$\cal C$ is a 1-parameter family of circles such that each point in $D$ is contained in two such circles.

Let us restate this as a theorem. What is remarkable here is that one uses exactly the usual Szeg\H{o} projection on each circle. 
\begin{theorem} Let $d\omega$ be the pushdown to $G$ of $dV$ on $\partial D$. Suppose $f\in L^2(G,d\omega)$.
Let $T$ be the operator whose output $Tf(\zeta)$ is the average of the Szeg\H{o} projections on the two circles from $\cal C$ containing $\zeta$ . Let $B_{2,d\omega} $ be the Bergman projection for $L^2(G,d\omega)$. Then $T^n\rightarrow B_{2,d\omega}$ in the strong operator topology. 

\end{theorem}

This theorem raises the following natural question. Consider the two curve families in the unit disc given by $|z|=r$ and 
$|\frac{z-a}{1-\overline{a}z}|, |a|<1$. 
Can one reconstruct weighted Bergman projections on the disc with a Bergman-Szeg\H{o} method?
This depends  on proving the analogue of the strip theorem---that functions which extend holomorphically from  each circle in these two curve families must be holomorphic. 
The author has  a proof for $L^{\infty}$ functions, and it probably can be extended to the case of $L^p, p>2$. However, the issue with applying the Lewy method falls apart at $p=2$ and this cannot be remedied by changing the weight of integration. The different geometry of the circle families at the boundary of the disc gets in the way.

For all theorems of this type, one possibility which we have not explored yet is to use $L^pL^q$ spaces, where you integrate over the curves in $L^q$ and then over the parameter in $L^p$. 
Perhaps one needs a different method entirely.
If one considers the original strip theorem without the modified weight, if it fails at $p=2$ there has to be a counterexample. 
That would be very interesting but probably the theorem is true with some better method of proof. 

\section{Assorted results and questions about operators associated to the strip problem}

Let $C$ be a curve for which the strip theorem holds as in \cite{La}. 
In addition, assume that $C$ is symmetric about the $x$-axis  and has a left right symmetry.
Let $f\in L^2_{\alpha}(S)$. Suppose $f(z)=f(-z)$.
Then we evidently have that $||T_1(f)||_{2,\alpha}=||T_2(f)||_{2,\alpha}$.   Do $T_1$ and $T_2$ respect even and odd? If so then norm equality holds for all functions.

A different issue is the "partial conjugate operator". 
In this case we must work with $L^2(S,\frac{dxdy}{\sqrt{1-y^2}}$. 
For simplicity, work with the circle. Choose the conjugate operator $u\rightarrow \widetilde{u}$ such that $\widetilde{u}(0)=0$. We consider this as an operator on the circle. We can construct operators $W_1$ and $W_2$ which are the left and right harmonic conjugate operators on the strip. Here are some questions.
\begin{enumerate}
\item Is it true that $||W_1(f)||_2$  and $||W_2(f)||_2$ are comparable? 
\item What is the spectrum of $W_i$?

\end{enumerate}
The second question may be related to the the question of whether the strip theorem has an approximate version. There are various ways to formulate a conjecture. Here is a simple version which assumes the strip theorem is true without the heavier weight---i.e. as in \cite{La3}.

\noindent{\bf Question:} Is there a way to estimate the distance to the Bergman space on $S$ with weight $\frac{dxdy}{\sqrt{1-y^2}}$ by integrating the distance to the Hardy space $H_2$ on each circle?
A positive answer to this question might have applications. 
\cite{La3}

One can say some concrete things about annihilators as well. We consider the case of the circle only, because the presentation is simpler and the results are optimal.

The strip theorem with the moving circle is true for 
$f\in L^p(S,\frac{dxdy}{\sqrt{1-y^2}}), p>2$. 
Given $L^p(S,\frac{dxdy}{\sqrt{1-y^2}})$ for any $1<p<\infty$ we can construct annihilators which are adapted to the strip problem.
Given $h(t,z)\in L^q({\bf R}\times S^1)$ with $h(t,\dot)\in H^q_0(S^1)$ (the 0 subscript denoting that the associated holomorphic function vanishes at 0), let $\widetilde{h}(z)=h(t_1,z)+h(t_2,z)$. Then 
$$\int_S f(z)\widetilde{h}(z)\frac{dx dy}{\sqrt{1-y^2}}=$$
$$\int_0^{\infty}\int_{S^1}f(t,z)h(t,z)d\theta dt=0.$$
Here $f(t,z)$ is the natural pullback of $f$ to the circle in the product.
We call the functions of the type $\widetilde h$ "strip annihilators".
The following proposition is an easy consequence of the strip theorem, by duality. Let $A_p(S)$ denote the $L^p$ Bergman space with weight $\frac{dxdy}{\sqrt{1-y^2}}$
\begin{prop} If $2<p<\infty$ and $q$ is the conjugate exponent, then the set of strip annihilators is dense in 
the space of annihilators to  $A_p(S)$.
\end{prop}

Now consider the annihilators in $L^q$ with $q>2$.

\begin{prop} For $q>2$, the map $h\rightarrow \widetilde{h}$ is 1-1.
\end{prop}

\begin{proof}
Suppose that $h\rightarrow \widetilde{h}$ is zero for some nonzero $h$. This means that $h(t_1, z)=h(t_2,z)$ a.e.
Then $H(z)=(h(t_1,z))^2=(h(t_2,z))^2 $ is well defined on $S$ and satisfies the hypothesis of the strip theorem, making it actually analytic. We also have that $H$ vanishes on the real axis because the $h(t,\dot)$'s are annihilators. So $H$ is 0, which means $h$ also is 0. 
\end{proof}

What about the case $p=2$? If you knew the strip problem was true for the circle for $p=2$ without the heavier weight, then we could say the map $h\rightarrow \widetilde{h}$ is 1-1 with dense image. 
If the image were closed we would have strong consequences.
That may be too strong, but there are two possibilities.
Maybe if you change the weight in the image space, you can get a closed range; alternatively, there might be an $L^p$ to $L^q$ estimation of some sort.

\bibliographystyle{plain}

\end{document}